\documentclass[11pt,a4paper]{article}
\usepackage{empheq}
\usepackage{amssymb,amsbsy,amsmath,amsfonts,amssymb,amscd,amsthm, mathrsfs, bbm}
\usepackage{amssymb}
\usepackage{color}

\usepackage{hyperref}

\newcommand{\RR}{\mathbb{R}}

\newtheorem{theo}{Theorem}

\newtheorem{rem}[theo]{Remark}

\newcommand{\beqn}{\begin{equation}}
\newcommand{\eeqn}{\end{equation}}
\newcommand{\bear}{\begin{eqnarray}}
\newcommand{\eear}{\end{eqnarray}}
\newcommand{\bean}{\begin{eqnarray*}}
\newcommand{\eean}{\end{eqnarray*}}

\date{}
\begin{document}

\title{On the linearized  system of equations for the condensate-normal fluid interaction at very low temperature.}
\maketitle
\begin{center}
\vskip -1cm 
{\large M. Escobedo}\\
\vskip 0.2cm 
{Departamento de Matem\'aticas, \\Universidad del
Pa{\'\i}s Vasco, (UPV/EHU)\\ Apartado 644, E--48080 Bilbao, Spain.\\
E-mail~: {\tt mtpesmam@lg.ehu.es}}
\end{center}

\begin{abstract}
The linearization around one of its equilibrium of a system  that describes the correlations between the superfluid component and the normal fluid part of a condensed Bose gas in the approximation of very low temperature and small condensate density, is studied. A simple and transparent argument gives a necessary and sufficient condition on the initial data  for the existence of global solutions  satisfying  the conservation of the total number of particles and energy.  Their convergence to a suitable stationary state is also shown and rates of convergence for the normal fluid and superfluid components are obtained.
\vskip 0.2cm 
\noindent
Keywords: BE condensate, excitations, global existence, rates of convergence.
\vskip 0.2cm 
\noindent
MSC[2008]: 45K05, 45A05, 45M05, 82C40, 82C05, 82C22.
\end{abstract}

\null \hskip 0.3cm Keywords: BE condensate, excitations, global existence, rates of convergence.

\null \hskip 0.3cm  MSC[2008]: 45K05, 45A05, 45M05, 82C40, 82C05, 82C22

\section{Introduction}\label{Introduction}
In a uniform condensed Bose gas at temperature below  the condensation temperature $T_c$,  the correlations between the superfluid component and the normal fluid part, corresponding to the excitations, may be described by the following equation
\begin{align}
&\frac {\partial n} {\partial t}(t, p)=Q_3\Big(n_c(t), n(t)\Big)(p),\quad t>0,\; p\in \RR^3, \label{PA}\\
&Q_3\Big(n_c, n\Big)(p)=\iint _{(\RR^3)^2}\!\!\big[R(p, p_1, p_2)\!-\!R(p_1, p, p_2)\!-\!R(p_2, p_1, p) \big]\,dp_1dp_2. \label{E1BCD}\\
&R(p, p_1, p_2)\,=|\mathcal M(p, p_1, p_2)|^2 \left[\delta ( \omega -\omega_1- \omega_2 )\,  \delta (p-p_1-p_2)\right]\times \nonumber \\
&\hskip 7cm  \times \left[ n_1n_2(1+n)-(1+n_1)(1+n_2)n \right],  \label{S1EA4BEJ}
\end{align}
(\cite{Eckern, Kirkpatrick,  Za}), where $n(t, p)$ represents the density of particles in the normal gas that at time $t>0$ have momentum $p$ and energy $\omega $.  The equation (\ref{PA}) is complemented with the equation for the fluctuation of the condensate density,
\begin{align}
\frac {d n_c(t)} {dt}=\int  _{ \RR^3 }Q_3\Big(n_c(t), n(t)\Big)(p)dp,\,\,t>0. \label{PANc}
\end{align}
We consider in this letter the approximation of the Bogoliubov dispersion law for very low temperature and  large number density of condensed atoms. For dimensionless variables, in units which minimize the number of prefactors,
\begin{align}
&\omega=\left(n_c(t) |p|^2+|p|^4\right)^{1/2}\approx \sqrt{n_c(t)}\,|p|\label{S1EM}\\
&|\mathcal M(p, p_1, p_2)|^2\approx\frac {|p| |p_1||p_2|} {n_c(t)}\label{S1EMB}.
\end{align} 
(cf. \cite{D, Eckern, Za, Kirkpatrick}). The resulting equation in its wave turbulence version has been studied in detail in \cite{D} (see also \cite{ZN}), where properties of the principal part of the collision integral operator are presented. The non linear system (\ref{PA}), (\ref{PANc}) has been considered for radial solutions in the mathematical literature in the limit (\ref{S1EM}), (\ref{S1EMB})  in \cite{A}, where global existence of solutions are obtained under some conditions on the initial data. We consider here, with a slightly different approach, the non radial close to equilibrium situation, and the  long time behavior of the system described by the linearized equations.

The equation (\ref{PA}) has a family of equilibria $n_\mu (\omega ) =(e^{\omega -\mu }-1)^{-1}$ for $\mu \le 0$. It follows that, for any constant $\kappa>0$, the pair $(n_\mu (\omega ), \kappa)$ with $\omega =\sqrt \kappa\,  |p|$ is an equilibrium of the system (\ref{PA}), (\ref{PANc}).  Without any loss of generality the value  $\kappa=1$ is taken in all the sequel. The  linearized equation (\ref{PA}) around the equilibrium $n_0$ was first studied in \cite{Bu}, for radially symmetric perturbations and  constant function $n_c(t)$. Non radial perturbations were considered in \cite{EPV, ET}, under the same assumption on $n_c(t)$.  The linearization of equation (\ref{PA}) is obtained using the change of dependent variable
$$
n(t)=n_0+n_0(1+n_0)\, F (t)
$$
in the collision integral (\ref{E1BCD}) and keeping only linear terms in $F$ in the resulting equation (cf. for example  \cite{EPV, {ET}}). Under the approximation (\ref{S1EM}) the collision  manifold in the momentum space reduces to those $p$, $p_1$ and $p_2$ that are collinear.  The resulting equation is then,
\begin{align}
&n_0(p )[1+n_0(p )]\frac {\partial F} {\partial t}(t, p)=\frac {1} {n_c(t)}\mathscr L(F) \label{E50}\\
&\mathscr L(F)= -\Gamma (|p|)\,n_0(|p|)(1+n_0(|p|))F (t, p)+\int  _{ \RR^3 } F(t, p')\,W(p, p')dp'  \label{E51}\\
&\Gamma (x)=\sinh x\int _0^\infty \!\!\!  \frac {y^2} {\sinh y}\left(\frac {|x-y|^2} {\sinh |x-y|}+\frac {(x+y)^2} {\sinh (x+y)} \right)dy  
\end{align}
and where  the explicit expression of the function $W$, given in \cite{EPV, ET}, will not be needed.  When $n_c(t)$ is a constant function 
an existence and uniqueness result of global solutions $F(t, p) $ to  (\ref{E50})was proved in \cite{ET}, for all  data  $F_0 \in L^2(d\mu)$, was proved in  \cite{ET} where $d\mu\equiv n_0(1+n_0)dp$. It was also shown that if $\int  _{ |p|<1 }\frac {|\Omega_0 (p)|^2 d\mu } {|p|}<\infty$,
\begin{align}
\label{S1ERatedecay}
||F(t)-F _{ \infty } ||_{L^2(d\mu )}
\le  \frac{C}{(1+t)^{1/2 }}||F_0-F _{ \infty }||_{L^2(d\mu )  },\,\,\,\forall t>0,
\end{align}
where, if $\left\{ Y _{\ell\, m  }\right\} _{ \ell, m }$ denotes the spherical harmonics on $\mathbb{S}^2$
\begin{align}
F_\infty (p)=\left(\sum _{ \ell=0 }^\infty \sum _{ m=-\ell }^\ell c_{ \ell\, m }Y _{ \ell\, m }\left(\frac {p} {|p|} \right)\right) |p|;\,\,\,c _{ \ell\, m }=\int  _{ \RR^3 }F_0(p)Y _{ \ell\, m }\left(\frac {p} {|p|} \right)d\mu.  \label{S1ETeta}
\end{align}
A different extreme regime of a uniform condensed Bose gas,  for temperatures below but close to the critical temperature $T_c$  and with small but non constant  number density of condensed atoms is described in the literature of physics (\cite{D, Eckern, Za, Kirkpatrick}. The corresponding linearized equation around an equilibria has started to be studied in detail in \cite{m}.
\section{The linearized system for non constant $n_c(t)$.}
We wish to describe in this letter what happens when $n_c(t)$ is not a constant, and then equation (\ref{E50})  must be supplemented by the linearization of (\ref{PANc}).  Let us then consider, for the functions $u(t, p)$ and $m_c(t)$  the system,
\begin{align}
&n_0(p )[1+n_0(p )]\frac {\partial u } {\partial t}(t, p)=\frac {1} {m_c(t)}\mathscr L(u(t))(p) \label{E50B}\\
&\frac {d m_c(t)} {dt}=-\frac {1} {m_c(t)}\int _{ \RR^3 }\mathscr L(u (t) )(p)dp. \label{E56}
\end{align}
The coupled system (\ref{E50B}), (\ref{E56}) is not directly addressed, as it is in \cite{A} for the non linear case. Instead, as a first step,  the term $m_c(t)^{-1}$ in the right hand side of (\ref{E50B}) is absorbed by means of a change of time variable:
\begin{equation}
\label{ETCh}
\tau =\int _0^t\frac {ds} {m_c(s)}
\end{equation}
The equation (\ref{E50B}) for the new function $v(\tau , p)=u(t, p)$ is then reduced to 
\begin{align}
&n_0(p )[1+n_0(p )]\frac {\partial v } {\partial \tau }(\tau , p)=\mathscr L(v (\tau )) \label{E50b}
\end{align}
to which the results of \cite{ET} may be applied. It is then proved, in a second step that the change of variables (\ref{ETCh}) may be inverted and that function $m_c$ and $u$ may be deduced to satisfy (\ref{E50B}), (\ref{E56}). Our results are the following,
\begin{theo}\label{theorem1}
Suppose that  $u _0\in L^2(d\mu )$ 
and let $v $ be the solution of equation (\ref{E50}) with $n_c(t)\equiv 1$ and initial data $u_0$.
Suppose that $m_c(0)>0$ is such that
\begin{align}
\label{S2Epc0}
m_c^2(0)>2\int  _{ \RR^3 }\Big(v  (\tau , p)-u_0(p)\Big)d\mu ,\,\,\forall \tau >0,
\end{align}
and let $u_\infty$ be defined as $F _{ \infty }$ in (\ref{S1ETeta}) with $F_0$ replaced by $u_0$.
Then, there exists a unique pair of functions $(u (t, p), m_c(t))$ with
\begin{align}
u \in L^\infty \left(0,\infty;  L^2(d\mu )\right)\cap\, 
 &C \left([0,\infty); L^2(d\mu )\right);\,\,\,\,u - u_\infty \in L^2\left(0, \infty; L^2(d\mu )\right),\label{S1E250}\\
& \frac {\partial u } {\partial t}\in L^2\left(0, \infty;  L^2\left(\frac {d\mu } {\Gamma (|p|)}\right)\right)\label{S1E252}\\
&m_c\in C([0, \infty)),\,\,m_c(t)>0,\,\,\forall t>0,
\end{align}
satisfying the equation  (\ref{E50B}) in $L^2\left(0, \infty;  L^2\left(\frac {d\mu } {\Gamma (|p|)}\right)\right)$ and equation (\ref{E56}) for all $t>0$, such that:
\begin{align}
\label{S1Einitial1}
&\lim _{ t\to 0 }\left( |m_c(t)-m_c(0)| + ||u (t)-u _0|| _{ L^2\left(\frac {d\mu } {\Gamma (|p|)}\right)}+
||u (t)-u _0|| _{ L^2(d\mu )}\right)=0.
\end{align}
This solution also satisfies the following conservation properties:
\begin{align}
\label{consmassM}
&p'_c(t)+\frac {d} {dt}\int  _{ \RR^3 }n_0(p)(1+n_0(p))u (t, p)dp=0,\,\,\,\forall t>0\\
\label{consenergyM}
&\frac {d} {dt}\int  _{ \RR^3 }n_0(p)(1+n_0(p))u (t, p)|p|dp=0,\,\,\,\forall t>0.
\end{align}
\end{theo}
\begin{proof}
Let us define, for all $\tau >0$,
\begin{align}
\label{S2Emtau}
m(\tau )&=\int  _{ \RR^3 }\mathscr L(v(\tau , p))dp\\
\label{S2Eqctau}
q_c(\tau )&=\Bigg( m_c(0)^2-2\int _0^\tau m(\sigma )d\sigma \Bigg)^{1/2}.
\end{align}
By (\ref{S2Emtau}) and the equation (\ref{E50b}),
\begin{align}
\label{S2EL1}
\int _0^\tau m(\sigma )d\sigma&=\int _0^\tau \frac {d} {d\tau } \int  _{ \RR^3 }n_0(1+n_0)v(\sigma )dpd\sigma
=\int  _{ \RR^3 }n_0(1+n_0)(v(\tau  , p)-u_0(p))dp
\end{align}
and it follows from condition (\ref{S2Epc0}) that $q_c(\tau )$ is well defined for all $\tau >0$. We may then define
\begin{align}
\label{S2Et}
t=\int _0^\tau q_c(\sigma )d\sigma. 
\end{align}
Since $q_c(\tau )> 0$ for all $\tau >0$, the right hand side of (\ref{S2Et}) is an increasing function of $\tau $. Moreover, if we denote
$$
N_*=\int  _{ \RR^3 }n_0(1+n_0)(u_\infty (p)-u_0(p))dp,
$$
by the convergence property (\ref{S1ERatedecay}) applied to $v$,
\begin{align}
&\left|q_c(\tau )^2-m_c(0)^2+2N_*\right|\le 2\int  _{ \RR^3 }|v(\tau , p)-u_\infty(p)|d\mu  \nonumber\\
&\le 2||v (\tau )-u_\infty ||_{L^2(d\mu )  } ||n_0(1+n_0)||_1^{1/2}
\le \frac{C||u _0-u_\infty||_{L^2(d\mu )  }}{(1+\tau )^{1/2 }} \label{Epx}
\end{align}
for some  numerical  constant $C>0$. Therefore,
\begin{align}
\label{Epx0}
\lim _{ \tau \to \infty }q_c^2(\tau )=m_c^2(0)-2N_*>0,
\end{align}
using (\ref{S2Epc0}) again and it  follows that the function $q_c$ is not integrable on $(0, \infty)$. Then, given any $t>0$ there exists a unique $\tau >0$ satisfying (\ref{S2Et}). We define then,
\begin{align}
\label{Pex}
u(t)= v(\tau ),\,\,\,m_c(t)=q_c(\tau ).
\end{align}
By (\ref{S2Et}) and Theorem 1.1 in \cite{ET}, the equation
\begin{align*}
n_0(1+n_0)\frac {\partial u} {\partial t}=n_0(1+n_0)\frac {\partial v} {\partial \tau }\frac {\partial \tau } {\partial t}=\mathscr L(v(\tau ))\frac {1} {q_c(\tau )}
=\frac {1} {m_c(t)}\mathscr L(u(t ))
\end{align*}
is satisfied in $L^2\left(0, \infty;  L^2\left(\frac {d\mu } {\Gamma (|p|)}\right)\right)$. On the other hand, since by (\ref{S2Eqctau})
\begin{align*}
\frac {dm_c(t)} {dt}&=\frac {dq_c(\tau )} {d\tau }\frac {d\tau } {dt }=-\frac {2 m(\tau )} {2q_c(\tau )}\frac {d\tau } {dt}\\
&=-\frac { m(\tau )} {m_c(t )}\frac {d\tau } {dt}=-\frac {1} {m_c(t)}\frac {d} {dt}\int  _{ \RR^3 }\mathscr L(u(t))dp,\,\,\,\forall t>0,
\end{align*}
equation (\ref{consenergyM}) is satisfied for all $t>0$. Properties (\ref{S1E250})-(\ref{S1Einitial1}) and (\ref{consenergyM}) are  consequences of properties
(15)-(19)  of Theorem 1.1 in \cite{ET} applied to the function $v(\tau )$ and, for Properties  (\ref{S1E252}) and (\ref{consenergyM}), that 
$\partial  _t=q_c(\tau )^{-1}\partial _\tau $ with $q_c\in C(0, \infty)$, $q_c(t)>0$ for all $t>0$. On the other hand, by Theorem 1.1 in \cite{ET}
\begin{align*}
\frac {d} {dt}\int  _{ \RR^3 }n_0(1+n_0)u(t, p)dp&=\frac {1} {m_c(t )}\frac {d} {d\tau }\int  _{ \RR^3 }n_0(1+n_0)v(\tau , p)dp=\frac {1} {m_c(t )}\int  _{ \RR^3 }n_0(1+n_0)\frac {v(\tau , p)} {\partial \tau}dp\\
&=\frac {1} {m_c(t )}\int  _{ \RR^3 }\mathscr L(v(\tau )(p)dp
=\frac {1} {m_c(t )}\int  _{ \RR^3 }\mathscr L(u(t )(p)dp
\end{align*}
and by (\ref{consenergyM}) the conservation (\ref{consmassM}) is satisfied.

If $(\tilde u,\tilde m_c)$ is another solution satisfying (\ref{S1E250})-(\ref{S1Einitial1}), the function $\tilde v(\tau , p)=\tilde u(t, p)$ for $\tau $ given by (\ref{ETCh}) would be a solution to (\ref{E50b}) satisfying (\ref{S1E250})-(\ref{S1E252}) and (\ref{S1Einitial1}). It then follows by Theorem 1.1 in \cite{ET} that $\tilde v(\tau )=v(\tau )$. Then,, since  $\tilde u(t, p)=\tilde v(\tau , p)$ for almost every  $p\in \RR^3$,
\begin{align*}
\frac {d \tilde m_c(t)} {dt}=-\frac {1} {\tilde m_c(t)}\int _{ \RR^3 }\mathscr L(\tilde u (t) )(p)dp=
-\frac {1} {\tilde m_c(t)}\int _{ \RR^3 }\mathscr L(\tilde v (\tau ) )(p)dp=-\frac {1} {\tilde m_c(t)}\int _{ \RR^3 }\mathscr L( v (\tau ) )(p)dp.
\end{align*}
Moreover $v(t, p)=u(t , p)$ for almost every  $p\in \RR^3$ and then
\begin{align}
\label{SEunidf}
\frac {d \tilde m_c(t)} {dt}=-\frac {1} {\tilde m_c(t)}\int _{ \RR^3 }\mathscr L( u (t ) )(p)dp.
\end{align}
By the Lipschitz property of right hand side of the equation  (\ref{SEunidf}) we deduce $\tilde m_c(t)=m_c(t)$ for all $t>0$ and this shows the uniqueness of the pair $(u, m_c)$ satisfying (\ref{E50B}), (\ref{E56}) and (\ref{S1E250})-(\ref{S1Einitial1}).
\end{proof}
\begin{theo}
\label{theorem2}
Let $u_0$ and $m_c(0)$ be as in Theorem \ref{theorem1}. 
If $u _0$ satisfies also:
\begin{align}
\label{S1Econdition}
\int  _{ |p|<1 }|u (p)|^2\, |p|^{-1}d\mu <\infty
\end{align}
then, there exists a constant $C=C(u_0)$ such that for all $t>0$,
\begin{align}
\label{S1ERatedecay2}
||u (t )-u_\infty ||_{L^2(d\mu ) }+\left|m_c^2(t) -m_c^2(0)+2N_*\right|\le 
\frac {C||u _0-u_\infty||_{L^2(d\mu )}} {\left(1+\frac {t} {(m_c^2(0)-2N_*)} \right)^{1/2}}.
\end{align}
\end{theo}
\begin{proof}
By (\ref{Epx}) and (\ref{Pex}),
\begin{align}
&\left|m_c(t )^2-m_c(0)^2+2N_*\right|\le \frac{C||u _0-u_\infty||_{L^2(d\mu )}}{(1+\tau(t) )^{1/2 }} \label{Epx2}\\
&\hbox{where}\,\,\,\tau (t)=\int _0^t\frac {ds} {m_c(s)}, \label{Epx3}
\end{align}
and by (\ref{Epx0}) and (\ref{Pex}),
$$
\lim _{ t\to \infty }m_c(t)=(m_c^2(0)-2N_*)^{1/2}.
$$
It follows that there exists $t_*>0$ such that,
\begin{equation*}
2 m_c(s)\ge (m_c^2(0)-2N_*)^{1/2},\,\,\forall s>t_*
\end{equation*}
and
\begin{align}
\label{Etauinf}
\tau (t)\ge \int  _{ t_* }^t\frac {ds} {m_c(s)}>\frac {\sqrt 2 (t-t_*)} {(m_c^2(0)-2N_*)^{1/2}},\,\,\forall t>t_*.
\end{align}
Then, for some constant $C>0$,
\begin{align}
\left|m_c^2(t) -m_c^2(0)+2N_*\right|\le \frac {C||u _0-u_\infty||_{L^2(d\mu )  }} {\left(1+\frac {t} {(m_c^2(0)-2N_*)} \right)^{1/2}},\,\,\,\forall t>0.
\end{align}
On the other hand, again  by (\ref{S1ERatedecay}) applied to $v$, and  (\ref{Etauinf}),
\begin{align*}
||u (t )-u_\infty ||_{L^2(d\mu ) }&=||v (\tau )-u_\infty ||_{L^2(d\mu ) }\le  \frac{C||u _0-u_\infty||_{L^2(d\mu ) }}{(1+\tau(t) )^{1/2 }}\le 
\frac{C||u _0-u_\infty||_{L^2(d\mu)   }}{(1+t )^{1/2 }}
\end{align*}
and (\ref{S1ERatedecay2}) follows. \end{proof}
\begin{rem}  (i) The non local condition (\ref{S2Epc0}) is necessary in order to have a global solution $(u, m_c)$ of (\ref{E50B}), (\ref{E56}) satisfying (\ref{S1E250})-(\ref{S1Einitial1}). If such a solution exists indeed,  the function $m_c$ must satisfy  $m_c^2(t)=m_c^2(0)-2\int  _{ \RR^3 }(v  (\tau , p)-u_0(p))d\mu $ for all $t>0$ where it is defined, and so (\ref{S2Epc0}) must be fulfilled.  

(ii) It would be interesting to know if there exists initial data $u_0$ for which as $t\to \infty$, $m_c(t)\to 0$  or $\int  _{ \RR^3 }u_\infty(p)d\mu =0$.
\end{rem}

\noindent
\textbf{Acknowledgments.}
The research of the author is supported by grants PID2020-112617GB-C21 of MINECO and IT1247-19 of the Basque Government.

 \end{document}